\documentclass[12pt,reqno]{article}

\usepackage[usenames]{color}
\usepackage{amssymb}
\usepackage{amsmath}
\usepackage{amsthm}
\usepackage{amsfonts}
\usepackage{amscd}
\usepackage{graphicx}

\usepackage[colorlinks=true,
linkcolor=webgreen,
filecolor=webbrown,
citecolor=webgreen]{hyperref}

\definecolor{webgreen}{rgb}{0,.5,0}
\definecolor{webbrown}{rgb}{.6,0,0}

\usepackage{color}
\usepackage{fullpage}
\usepackage{float}

\usepackage{graphics}
\usepackage{latexsym}
\usepackage{epsf}
\usepackage{breakurl}

\def\modd#1 #2{#1\ \mbox{\rm (mod}\ #2\mbox{\rm )}}

\usepackage[linesnumbered,lined,boxed,commentsnumbered]{algorithm2e}

\begin{document}

\theoremstyle{plain}
\newtheorem{theorem}{Theorem}
\newtheorem{corollary}[theorem]{Corollary}
\newtheorem{lemma}[theorem]{Lemma}
\newtheorem{proposition}[theorem]{Proposition}

\theoremstyle{definition}
\newtheorem{definition}[theorem]{Definition}
\newtheorem{example}[theorem]{Example}
\newtheorem{conjecture}[theorem]{Conjecture}

\theoremstyle{remark}
\newtheorem{remark}[theorem]{Remark}
\begin{center}
\vskip 1cm{\Large\bf 
On Positive Integers $n$ with $\phi(n)=\frac{2}{3} \cdot (n+1)$
}
\vskip 1cm
\large
Christian Hercher\\
Institut f\"{u}r Mathematik\\
Europa-Universit\"{a}t Flensburg\\
Auf dem Campus 1c\\
24943 Flensburg\\
Germany \\
\href{mailto:christian.hercher@uni-flensburg.de}{\tt christian.hercher@uni-flensburg.de} \\
\end{center}

\vskip .2 in
\begin{abstract}
    While solving a special case of a question of Erd\H{o}s and Graham Steinerberger asks for all integers $n$ with $\phi(n)=\frac{2}{3} \cdot (n+1)$. He discovered the solutions $n\in\{5, 5 \cdot 7, 5\cdot 7\cdot 37, 5\cdot 7\cdot 37\cdot 1297\}$ and found that any additional solution must be greater than $10^{10}$. He conjectured that there are no such additional solutions to this problem.

    We analyze this problem and prove:
    \begin{itemize}
        \item Every solution $n$ must be square-free.
        \item IF $p$ and $q$ are prime factors of a solution $n$ then $p\nmid (q-1)$.
        \item Any solution additional to the set given by Steinerberger has to have at least 7 prime factors.
        \item For any additional solution it holds $n\geq 10^{14}$.
    \end{itemize}
\end{abstract}

\section{Introduction}
In \cite{Erdos80} Erd\H{o}s and Graham define for positive integers $x$ the functions $g_0(x):=x+\phi(x)$ and $g_{k+1}(x):=g_0(g_k(x))$. They ask for solutions of $g_{k+r}(x)=2g_k(x)$. This is listed as Problem~\#411 in Bloom's list \cite{Bloom411}. For $r=2$ Steinerberger \cite{Steinerberger} gives six infinite sequences of solutions and postulate that these are all solutions. He proofs that every other solution has to be of the form $x=2^\ell \cdot p$ with a prime $p>10^{10}$ fulfilling $\phi((3p-1)/4)=(p+1)/2$. This criteria can be reformulated as the search for primes $p=8m+7$ such that $\phi(6m+5)=4m+4$, or---to put it in another way---  as the search for positive integers $n=6m+5$ with $\phi(n)=\frac{2}{3}\cdot (n+1)$
and $\frac{1}{3} \cdot (4n+1)$ being a prime. We drop the primality condition and ask for all positive integers $n$ with \begin{align}
\phi(n)&=\frac{2}{3}\cdot (n+1).\label{Grundgleichung}
\end{align}
Steinerberger finds four solutions of \eqref{Grundgleichung}, namely $n\in\{5,35,1195,1679615\}$, and conjectures that this list is complete. He checked and found no other solution under $10^{10}$. We investigate this question further.

\subsection{Outline}
In Section~\ref{sec:square-free} we prove that every ~$n$ solving \eqref{Grundgleichung} has to be square-free (Theorem~\ref{Thm:square-free}) and that for every two primes $p$ and $r$ dividing $n$ we have $p\nmid (r-1)$ (Corollary~\ref{Cor:prdiv}). 

Then, in Section~\ref{sec:SmallSets} we prove directly that Steinerbergers list of known solutions is complete, if one only considers positive integers with at most four prime factors (Theorem~\ref{Thm:k<=4}). We also show that for every fixed number of prime factors there are at most finitely many solutions (Corollary~\ref{Cor:Finiteness}).

And in Section~\ref{sec:Program} we describe how we used the insides gained in the previous sections to write a fast computer program that automates this search for further solutions of \eqref{Grundgleichung}. We state the findings that there are no such solutions with $n$ having exactly five or six prime factors (Theorem~\ref{Thm:k<=4}) and that for every additional solution $n$ has to be at least $10^{14}$ (Theorem~\ref{Thm:nmax}). 

\section{Any solution is square-free} \label{sec:square-free}
\begin{theorem} \label{Thm:square-free}
    Let $n$ be a positive integer with $\phi(n)=\frac{2}{3}\cdot (n+1)$. Then $n$ is square-free.
\end{theorem}

\begin{proof}
We want to find all positive integers $n$ which solve equation \eqref{Grundgleichung}. Obviously, $n=1$ is not a solution. Thus, from now on let be $n>1$. As 
\begin{align}
\phi(n)&=n\cdot \prod_{p\mid n} \frac{p-1}{p} \nonumber\\
\intertext{\eqref{Grundgleichung} is equivalent to} 
\prod_{p\mid n} \frac{p-1}{p} &= \frac{2}{3} \cdot \left(1+\frac{1}{n}\right) \label{(2)}
\intertext{or}
n&=\frac{2 \prod_{p\mid n} p}{3 \cdot \prod_{p\mid n} (p-1) - 2\prod_{p\mid n} p}=:\frac{N}{D}. \label{(3)}
\end{align}
First observe that for every $p\mid n$ the left hand side of \eqref{(2)} is $\leq 1-\frac{1}{p}$. Thus, would $n$ be divisible by 2 or 3, the product on the left side of~\eqref{(2)} would be at most $1-\frac{1}{3}=\frac{2}{3}$, hence smaller than the right hand side of~\eqref{(2)}. Thus, we know for all primes $q\mid n$ that $q\geq 5$. 

Any such prime divisor $q$ of $n$ clearly divides the left hand side of~\eqref{(3)} and the numerator~$N$ of the right hand side of~\eqref{(3)}, but $N$ only once. So it does not divide the denominator~$D$. Thus, the only possible non-trivial common divisor of $N$ and $D$ is~2. Since $n$ must be divisible by at least one odd prime $D$ is even, hence $\mathop{gcd}(D, N)=2$. Now $n$ is an integer, which means that $D \mid N$. Hence, $D \mid 2$. Since also $2=\mathop{gcd}(D,N)\mid D$ and $D=\frac{n}{N}>0$ it follows $D=2$. Thus,
\[3 \cdot \prod_{p\mid n} (p-1) - 2\prod_{p\mid n} p=2\]
and $n=\prod_{p\mid n} p$ is square-free.
\end{proof}

Now we can reformulate the problem: Find all finite sets $\emptyset\neq Q\subset \mathbb{P}\setminus\{2,3\}$ of primes with
\begin{align}
3 \cdot \prod_{q\in Q} (q-1) - 2\prod_{q\in Q} q&=2.
\label{NewFormulation}
\intertext{In this setting we have}
\prod_{q\in Q} q&=n.
\end{align}

Remark that with $n=\prod_{q\in Q}$ it follows from \eqref{NewFormulation} that $2\equiv 0-2n\equiv n\pmod{3}$, hence $n\equiv 5\pmod{6}$. So the restriction to such $n$ by Steinerberger does not loose any other solutions of \eqref{Grundgleichung}.

\begin{corollary}\label{Cor:prdiv}
    Let $n$ be a positive integer with $\phi(n)=\frac{2}{3}\cdot (n+1)$ and let $p$ and $r$ be two primes with $pr\mid n$. Then $p\nmid (r-1)$.
\end{corollary}

\begin{proof}
    From equation~\eqref{NewFormulation} we now that there is a set $\{p,r\}\subseteq Q\subset \mathbb{P}\setminus\{2,3\}$ of primes with $n=\prod_{q\in Q} q$. Obviously, $p$ divides $n$ and $r-1$ divides $\prod_{q\in Q} (q-1)$, but from ~\eqref{NewFormulation} we know that both products are relatively prime, because $n$ is odd. Hence, $p\nmid (r-1)$.
\end{proof}

\section{Solutions for small sets with \texorpdfstring{$|Q|\leq 4$}{|Q|<=4}} \label{sec:SmallSets}
From now on let $\mathbb{P}=\{p_1,p_2,\dots\}$ the set of primes with $2=p_1<p_2<\dots$ and $\emptyset \neq Q\subset \mathbb{P}$ a finite subset with $Q=\{q_1,\dots,q_k\}$ and $5\leq q_1 < \dots < q_k$. 

\subsection{The case \texorpdfstring{$k=|Q|=1$}{k=|Q|=1}}
\begin{lemma}\label{Lemma:k=1}
    The only solution of equation~\eqref{NewFormulation} with $k=1$ prime factor is $n=5$.
\end{lemma}

\begin{proof}
In the case of $k=1$ equation~\eqref{NewFormulation} reduces to $3(q_1-1)-2q_1=2$, or $q_1=5$. Thus, $n=q_1=5$ is the only solution for \eqref{(2)} with $n$ only having one prime factor.
\end{proof}

\subsection{The case \texorpdfstring{$k=|Q|=2$}{k=|Q|=2}}\label{sec:k=2}
\begin{lemma}\label{Lemma:k=2}
    The only solution of equation~\eqref{NewFormulation} with $k=2$ prime factor is $n=5 \cdot 7$.
\end{lemma}

\begin{proof}
In the case of $k=2$ equation~\eqref{NewFormulation} reduces to 
\begin{align*}
    3(q_1-1)(q_2-1)-2q_1q_2&=2.\\
    \intertext{This is equivalent to}
    q_1q_2-3q_1-3q_2+1&=0\\
    (q_1-3)(q_2-3)&=8.
\end{align*}
Since $q_1$ and $q_2$ are both odd, the only factorization valid is $q_1-3=2$ and $q_2-3=4$, thus $q_1=5$, $q_2=7$, and $n=q_1q_2=35$. This is is the only solution for \eqref{(2)} with two prime factors.
\end{proof}

\subsection{A generalized variant of the case \texorpdfstring{$k=2$}{k=2}}\label{sec:k2general}
The method of factoring a certain term to find all solutions of a problem which was used in the previous Section~\ref{sec:k=2} can be generalized:

\begin{lemma}\label{Lemma:k=2factoring}
Let $A>B\geq 2$ be positive integers and $q_1<q_2$ be two primes with 
\begin{align}
A\cdot (q_1-1) \cdot(q_2-1)&=B\cdot q_1 \cdot q_2+2.\label{k=2general}
\intertext{Then}
((A-B)q_1-A) \cdot ((A-B)q_2-A)&=AB+2(A-B).\label{k=2}
\intertext{If $A=2a$ and $B=2b$ are both even it follows}
((a-b)q_1-a)((a-b)q_2-a)&=ab+a-b. \label{k=2reduziert}
\end{align}
\end{lemma}

\begin{proof}
Let $A, B, q_1, q_2$ be as described in the lemma. We are interested in solutions of
\begin{align*}
    A\cdot (q_1-1) \cdot(q_2-1)&=B\cdot q_1 \cdot q_2+2.
    \intertext{As in the previous section we equivalently restate this equation as}
    (A-B) \cdot q_1q_2 - A \cdot (q_1+q_2) +A -2 &=0 \\
    (A-B)^2 q_1q_2 - A \cdot (A-B)(q_1+q_2) + (A-B)(A-2)&=0\\
    ((A-B)q_1-A) \cdot ((A-B)q_2-A)&=A^2-(A-B)(A-2) \\
    &=AB+2(A-B). 
    \intertext{If $A=2a$ and $B=2b$ this can be simplified further to}
    ((a-b)q_1-a)((a-b)q_2-a)&=ab+a-b.
\end{align*}
In the previous Section~\ref{sec:k=2} we had $A=3$ and $B=2$, in future settings the second variant will become useful. Hence, in the general situation we can find all solutions by factoring $AB+2(A-B)$ or $ab+a-b$, respectively. Since both factors on the left side of \eqref{k=2reduziert} are congruent to $-a \pmod{a-b}$ only factorizations $ab+a-b=f_1 \cdot f_2$ with $f_1\equiv -a \pmod{a-b}$ have to be considered. (Then it is automatically $f_2\equiv -a \pmod{a-b}$, too.) The same is true for equation \eqref{k=2}.
\end{proof}

\subsection{The case \texorpdfstring{$k=|Q|=3$}{k=|Q|=3}}
\begin{lemma}\label{Lemma:k=3}
    The only solution of equation~\eqref{NewFormulation} with $k=2$ prime factor is $n=5 \cdot 7 \cdot 37$.
\end{lemma}

\begin{proof}
In the case of $k=3$ equation~\eqref{NewFormulation} reduces to 
\begin{align*}
    3(q_1-1)(q_2-1)(q_3-1)-2q_1q_2q_3&=2.\\
    \intertext{With $n=q_1q_2q_3$ this is equivalent to}
    3 \cdot \prod_{i=1}^3 \left(1-\frac{1}{q_i}\right)&=2+\frac{2}{n}.
\end{align*}
If $q_1>5$ holds, we have $q_1\geq 7$, $q_2\geq 11$, and $q_3\geq 13$. Thus, the left hand-side of this equation is equal to or greater than $ 3\left(1-\frac{1}{7}\right) \cdot \left(1-\frac{1}{11}\right) \cdot \left(1-\frac{1}{13}\right)=\frac{2160}{1001}>2.15$. But on the right hand-side, we have $2+\frac{2}{n} \leq 2+\frac{2}{7\cdot 11 \cdot 13}<2.002$. Hence, there is no solution in this subcase and it must hold $q_1=5$.

If $q_1=5$, equation~\eqref{NewFormulation} reduces further to
\begin{align*}
    12(q_2-1)(q_3-1)-10q_2q_3&=2,
    \intertext{which is a special case of equation~\eqref{k=2general} with $A=12$, $B=10$, hence $a=6$, $b=5$, and $a-b=1$. Thus we can factorize it and get the equivalent equation}
    (q_2-6)\cdot (q_3-6)&=31.
\end{align*}
The only viable factorization is $q_2-6=1$, $q_3-6=31$, thus $q_1=5$, $q_2=7$, $q_3=37$, and $n=5\cdot 7 \cdot 37=1295$ is the only solution of \eqref{(2)} with three prime factors.
\end{proof}

\subsection{A generalized variant of the case \texorpdfstring{$k=3$}{k=3}---For every fixed \texorpdfstring{$|Q|$}{|Q|} there are at most finitely many solutions}
\begin{lemma}\label{Lemma:finitenessAB}
    Let $A>B\geq 2$ be positive integers and let $k\geq 1$ be a fixed integer. Furthermore, let $Q\subseteq \mathbb{P}\setminus\{2,3\}$ a set of primes with $|Q|=k$. Then the equation
    \begin{align}
    A\cdot \prod_{i=1}^k (q_i-1) &=B\cdot \prod_{i=1}^k q_i+2.\label{kprod_general}
    \end{align}
    has only finitely many solutions.
\end{lemma}

\begin{proof}
Let $A>B\geq 2$, $k$, and $Q$ as in the lemma. Further let $Q=\{q_1,\dots,q_k\}$ with $q_1<q_2<\dots<q_k$ primes. With $n=\prod_{i=1}^k q_i$ equation~\eqref{kprod_general} is equivalent to
\begin{align}
    \prod_{i=1}^k \left(1-\frac{1}{q_i}\right) &= \frac{B}{A} + \frac{2}{A\cdot n}\nonumber
    \intertext{and for even $A=2a$ and $B=2b$ to}
    \prod_{i=1}^k \left(1-\frac{1}{q_i}\right) &= \frac{b}{a} + \frac{1}{a\cdot n}. \label{kprodreduced}
\end{align}

Suppose $q_1>p_m$ for a positive integer~$m$. Then $q_i\geq p_{m+i}$ for all $1\leq i\leq k$, hence
\begin{align}
    \prod_{i=1}^k \left(1-\frac{1}{q_i}\right) &\geq \prod_{i=1}^k \left(1-\frac{1}{p_{m+i}}\right) \nonumber
    \intertext{and}
    \frac{b}{a} + \frac{1}{a\cdot n} &\leq \frac{b}{a} + \frac{1}{a\cdot \prod_{i=1}^k p_{m+i}}. \nonumber
    \intertext{Thus, if}
    \prod_{i=1}^k \left(1-\frac{1}{p_{m+i}}\right) &> \frac{b}{a} + \frac{1}{a\cdot \prod_{i=1}^k p_{m+i}} \label{finitenesscriterium}
    \intertext{there is no solution to equation~\eqref{kprodreduced}. And if}
    \prod_{i=1}^k \left(1-\frac{1}{p_{m+i}}\right) &> \frac{B}{A} + \frac{2}{A\cdot \prod_{i=1}^k p_{m+i}} \label{finitenessABcriterium}
\end{align}
there is no solution to equation~\eqref{kprod_general}.

So, only finitely many cases for $q_1$ remain to be considered. For every concrete value $q_1$ can assume one can substitute this value into equation~\eqref{kprod_general} and has reduced the number of variables by one. Now recursively in the same way $q_2$ can be bound, $\dots$, until one reaches a state with only two variables remaining were one can find all at most finitely many solutions via factorization as in Section~\ref{sec:k2general}.
\end{proof}

\begin{corollary}\label{Cor:Finiteness}
    For every fixed positive integer $k$ there are at most finitely many solutions of equation~\eqref{NewFormulation} with $n$ having exactly $k$ prime divisors.
\end{corollary}

\begin{proof}
Take $A=3$ and $B=2$ in Lemma~\ref{Lemma:finitenessAB}. 
\end{proof}

\subsection{The case \texorpdfstring{$k=|Q|=4$}{k=|Q|=4}}\label{sec:k=4}
\begin{lemma}\label{Lemma:k=4}
    The only solution of equation~\eqref{NewFormulation} with $k=4$ prime factor is $n=5 \cdot 7 \cdot 37 \cdot 1297$.
\end{lemma}

\begin{proof}
For $k=4$ equation~\eqref{NewFormulation} reduces to 
\begin{align}
   3\cdot \prod_{i=1}^4 (q_i-1) &=2\cdot \prod_{i=1}^4 q_i+2. \label{(6)}
   \intertext{Suppose $q_1>5=p_3$ holds. Then}
    \prod_{i=1}^4 \left(1-\frac{1}{p_{3+i}}\right) &=\frac{11520}{17017}=\frac{34560}{51051}\nonumber
   \intertext{and}
   \frac{2}{3} + \frac{2}{3\cdot \prod_{i=1}^k p_{3+i}}&=\frac{34036}{51051}.
   \intertext{Thus, the criterion in equation~\eqref{kprod_general} is fulfilled and there is no solution with $q_1\geq 7$. Hence, we can assume $q_1=5$. Then \eqref{(6)} transforms to}
    12\cdot \prod_{i=2}^4 (q_i-1) &=10\cdot \prod_{i=2}^4 q_i+2\nonumber
   \intertext{and}
    6\cdot \prod_{i=2}^4 (q_i-1) &=5\cdot \prod_{i=2}^4 q_i+1. \label{(7)}
    \intertext{Suppose now $q_2>13=p_6$. Then}
    \prod_{i=1}^3 \left(1-\frac{1}{p_{6+i}}\right) &=\frac{6336}{7429}\nonumber
    \intertext{and}
   \frac{5}{6} + \frac{1}{6\cdot \prod_{i=1}^k p_{6+i}}&=\frac{6191}{7429},\nonumber
\end{align}
  so the criterium in equation~\eqref{kprodreduced} is fulfilled and there is no solution with $q_2\geq 17$. Hence, we now can assume $5=q_1<q_2\leq 13$, in other words $q_2\in\{7,11,13\}$.

\subsubsection{The subcase \texorpdfstring{$q_2=13$}{q2=13}:}
If $q_2=13$ equation~\eqref{(7)} reduces further to
\begin{align*}
    72\cdot (q_3-1)(q_4-1) &=65\cdot q_3q_4+1.
    \intertext{As in Section~\ref{sec:k2general} we conclude that we then have (with $a=72$ and $b=65$, hence $a-b=7$)}
    (7q_3-72)(7q_4-72)&=4687=43\cdot 109.
\end{align*}
Since both divisors (1 and 43) of 4687 which are smaller than $\sqrt{4687}$ are not congruent to $-72\equiv 5\pmod{a-b}$ there is no solution in this subcase.

\subsubsection{The subcase \texorpdfstring{$q_2=11$}{q2=11}:}
If $q_2=11$ equation~\eqref{(7)} reduces further to
\begin{align*}
    60\cdot (q_3-1)(q_4-1) &=55\cdot q_3q_4+1.
\end{align*}
But $\mathop{gcd}(60,55)>1$. Hence, there is no solution in this subcase, either.

\subsubsection{The subcase \texorpdfstring{$q_2=7$}{q2=7}:}
If $q_2=7$ equation~\eqref{(7)} reduces further to
\begin{align*}
    36\cdot (q_3-1)(q_4-1) &=35\cdot q_3q_4+1.
    \intertext{As in Section~\ref{sec:k2general} we conclude that we then have (with $a=36$ and $b=35$, hence $a-b=1$)}
    (q_3-36)(q_4-36)&=1261=13\cdot 97.
\end{align*}
This gives to possible values for the smaller factor $(q_3-36)$: 1 and 13. In the second case $q_3$ would be $49=7^2$, hence not a prime which leads to no solution. In the first case both integers $q_3=1+36=37$ and $q_4=1261+36=1297$ are primes. Thus, this leads to the only solution $n=5\cdot 7\cdot 37 \cdot 1297=1679615$ with four prime factors.
\end{proof}
\subsection{Summarizing the cases \texorpdfstring{$k=|Q|\leq 4$}{k=|Q|<=4}}
Now we can combine:
\begin{theorem}\label{Thm:k<=4}
    The only solutions $n$ to equation~\eqref{NewFormulation} with at most four prime factors are
    \[5, 5\cdot 7, 5\cdot 7 \cdot 37, \text{ and } 5 \cdot 7 \cdot 37 \cdot 1297.\]
\end{theorem}

\begin{proof}
    This follows directly from Lemmas~\ref{Lemma:k=1}, \ref{Lemma:k=2}, \ref{Lemma:k=3}, and \ref{Lemma:k=4}.
\end{proof}

\section{Implementation and computer-assisted generated results }\label{sec:Program}

The work flow as demonstrated in Section~\ref{sec:k=4} for solving the case~$k=4$ can be generalized for larger values of~$k$ and automated. Thus, an implementation of these algorithms and the use of computers could lead to better results. With this in mind, we wrote a C++ program which uses the findings of in particular Theorem~\ref{Thm:square-free}, Corollary~\ref{Cor:prdiv}, and Lemmas~\ref{k=2general} and~\ref{Lemma:finitenessAB}. Some optimizations and adjustments for better calculations are being applied, in particular
\begin{itemize}
\item Since we are interested in more than one situation in these calculations, if a given integer is prime, we precalculate the set of the first million primes for quick lookups.
\item For further primalty testing in the first step trial division is implemented which is later supplemented by a Fermat probable prime test. To identify the composite integers~$n<2^{55}\approx 3.6\cdot10^{16}$ passing this test we use a scaled down version of the table of all Fermat pseudoprimes $<2^{64}$ given by Feitsma cite{Feitsma}. Since, we mostly do not need to consider larger integers this is sufficient. And, we excluded the pseudoprimes with small factors from this list, too, as we find them with the initial trial dividing.  
\item As the calculation comes down to factoring integers we utilize trial division and Pollard's $p-1$ factoring method. (As most integers we have to factor are $<10^{15}$ we do not bother to implement more sophisticated methods.
\item Instead of calculating products $\prod_{q\in Q} \left(1-\frac{1}{q}\right)$ directly we work with its logarithm: The factors $\left(1-\frac{1}{q}\right)$ are for larger values of ~$q$ nearly 1, and thus, can't be expressed with high precision. But $\log\left(1-\frac{1}{q}\right)$ tends to 0, hence the representation error when using floating point arithmetic is much lower.
\item To neglect rounding errors while using floating point arithmetic we use safe guards: If we have to guarantee that a value $a$ has to be smaller than a given limit~$L$, than we only use values of $a$ with $a<L \cdot (1-\epsilon)$ with a suitable (small) $\epsilon>0$.
\item As the different subcases during the computation are independent of another, we use OpenMP to parallelize these task on different CPU cores. (But this algorithm is not designed to use the parallelism of a GPU, since the algorithm branches in many ways. Thus, no SIMD (same instruction multiple data) constructs can be used.)
\end{itemize}

With these implementation we get

\begin{theorem} \label{Thm:k<=6}
    The are no solutions $n$ to equation~\eqref{NewFormulation} with exactly five or six prime factors. Thus, every additional solution, not already listed by Steinerberger, has to have at least seven different prime factors.
\end{theorem}

\begin{proof}
The program found no solutions with $k=5$ and $k=6$, respectively. (However, it reproduced the known solutions for $k\leq 4$.) The computation proceeded in a few seconds on a standard desktop PC.
\end{proof}

As the number of branches to consider increases by a large amount and a greater number of larger integers has to be factored a search for solutions with exactly $k=7$ prime factors is out of reach in a reasonable computing time. However, if one restricts the search and forces $n$ to be below a given limit, these tasks become feasible, again.

To do so we observe two small findings:

\begin{lemma}\label{Lemma:nmax}
Let $n=\prod_{i=1}^k q_i$ be a solution with $k\geq 3$, $5\leq q_1<\dots<q_k$ primes and $q_1,\dots,q_j$ already known. Furthermore, let $L\geq n$ be an upper limit on $n$. 
\begin{itemize}
\item Then we have 
\[q_{j+1} \leq \sqrt[k-j]{\frac{L}{q_1 \cdot \ldots \cdot q_j}}.\]
\item With $b=\prod_{i=1}^{k-2} q_i$ and $a=\frac{3}{2} \cdot \prod_{i=1}^{k-2} (q_i-1)$ it is
\[ab+a-b< (a-b)^2 \cdot \frac{L}{b}.\]
\end{itemize}
\end{lemma}

\begin{proof}
The first observation follows directly from the definition of $(q_i)$ as a monotonically increasing sequence.

For the second one we recall from equation~\eqref{NewFormulation} that 
\begin{align*}
3 \cdot \prod_{i=1}^{k} (q_i-1) &= 2 \cdot \prod_{i=1}^{k} q_i + 2
\intertext{holds. Set $A:=3 \prod_{i=1}^{k-2} q_i$ and $B:=2 \cdot \prod_{i=1}^{k-2} q_i$. Then both products are even and with $a:=\frac{A}{2}$ and $b:=\frac{B}{2}$ we get}
a \cdot (q_{k-1} - 1) \cdot (q_k-1) &= b \cdot q_{k-1}\cdot q_k + 1.
\intertext{From this we get as in Section~\ref{sec:k2general}}
((a-b)q_{k-1}-a) \cdot ((a-b)q_{k}-a)&=ab+a-b.
\intertext{As $a-b\geq 1>0$ and $a>0$ it is}
ab+a-b&<(a-b)^2 q_{k-1}q_k=(a-b)^2 \cdot \frac{n}{b}\leq (a-b)^2 \cdot \frac{L}{b}.
\end{align*}
\end{proof}

This can be used to cut branches in the computation:

\begin{itemize}
\item If at any point in the computation we get $q_{j+1}\leq q_j$ from Lemma~\ref{Lemma:nmax} then there is no solution in this case and it can be terminated.
\item If one has to factor an integer $ab+a-b$ in the algorithm which does not fulfill the second observation of  Lemma~\ref{Lemma:nmax} then this factorization cannot lead to a solution with $n\leq L$. Hence, we do not complete this factorization and cut this branch.
\end{itemize}

In this way we can exclude the existence of solutions $n\leq L$ with exactly $k$ prime divisors even for larger values of~$k$. But since $L \geq n\geq 5 \cdot 7 \cdot \ldots \cdot p_{k+2}$ for given $L$ the number $k$ cannot be arbitrarily large and is, in fact, limited by 12 in all our discussed calculations. With this we get in about two days of computing time on our desktop PC:

\begin{theorem}\label{Thm:nmax}
There is no additional solution with $n\leq 10^{14}$.
\end{theorem}

\bigskip
\hrule
\bigskip

\noindent 2020 {\it Mathematics Subject Classification}:
Primary 11A25; Secondary 1104.

\noindent \emph{Keywords: } Special value of phi function. 

\bigskip
\hrule
\bigskip

\end{document}